\documentclass[10pt, reqno]{amsart}

\usepackage{amsmath,amssymb,amsthm,amsfonts,verbatim}
\usepackage{microtype}
\usepackage[all,2cell]{xy}
\usepackage{mathtools}
\usepackage{graphicx}
\usepackage{pinlabel}
\usepackage{hyperref}
\usepackage{mathrsfs}
\usepackage{color}
\usepackage{enumerate}
\usepackage{cite}

\CompileMatrices

\usepackage[top=1.3in,bottom=1.9in,left=1.2in,right=1.2in]{geometry}

\theoremstyle{plain}
\newtheorem{theorem}{Theorem}[section]
\newtheorem{maintheorem}{Theorem}

\newtheorem{maincorollary}[maintheorem]{Corollary}

\newtheorem{proposition}[theorem]{Proposition}
\newtheorem{lemma}[theorem]{Lemma}
\newtheorem{conjecture}[theorem]{Conjecture}

\theoremstyle{definition}
\newtheorem{definition}[theorem]{Definition}

\newtheorem{remark}[theorem]{Remark}

\newcommand{\nc}{\newcommand}
\nc{\dmo}{\DeclareMathOperator}

\nc{\Q}{\mathbb{Q}}
\nc{\F}{\mathbb{F}}
\nc{\R}{\mathbb{R}}
\nc{\Z}{\mathbb{Z}}
\nc{\C}{\mathbb{C}}
\nc{\Ell}{\mathcal{L}}
\nc{\M}{\mathcal{M}}
\nc{\K}{\mathcal{K}}
\nc{\I}{\mathcal{I}}
\nc{\U}{\mathcal U}
\nc{\disk}{\mathbb{D}}
\nc{\hyp}{\mathbb{H}}

\nc{\CP}{\mathbb{CP}}
\nc{\cS}{\mathcal{S}}
\dmo{\Mod}{Mod}
\dmo{\PMod}{PMod}
\dmo{\LMod}{LMod}
\dmo{\Diff}{Diff}
\dmo{\Homeo}{Homeo}
\dmo{\dist}{dist}
\dmo\BDiff{BDiff}
\dmo\SO{SO}
\dmo\Hom{Hom}
\dmo\SL{SL}
\dmo\Sp{Sp}
\dmo\rank{rank}
\dmo\sig{sig}
\dmo\Out{Out}
\dmo\Aut{Aut}
\dmo\Inn{Inn}
\dmo\GL{GL}
\dmo\PSL{PSL}
\dmo\BHomeo{BHomeo}
\dmo\EHomeo{EHomeo}
\dmo\EDiff{EDiff}
\nc\Sig{\Sigma}
\dmo\Teich{Teich}
\dmo\Fix{Fix}
\nc{\pair}[1]{\langle #1 \rangle}
\nc{\abs}[1]{\left| #1 \right|}
\nc{\action}{\circlearrowright}
\nc{\norm}[1]{\left | \left | #1 \right | \right |}
\nc{\abcd}[4]{\left(\begin{array}{cc} #1 & #2 \\ #3 & #4 \end{array}\right)}
\nc{\into}\hookrightarrow
\dmo{\Isom}{Isom}
\nc{\normal}{\vartriangleleft}
\dmo{\Vol}{Vol}
\dmo{\im}{Im}
\dmo{\Push}{Push}
\dmo{\Conf}{Conf}
\dmo{\PConf}{PConf}
\dmo{\id}{id}
\dmo{\Jac}{Jac}
\dmo{\Pic}{Pic}
\dmo{\Stab}{Stab}
\dmo{\Arf}{Arf}
\dmo{\End}{End}
\dmo{\Gal}{Gal}
\dmo{\lcm}{lcm}
\dmo{\ab}{ab}
\dmo{\opp}{op}
\dmo{\SU}{SU}
\dmo{\tr}{tr}

\nc{\Span}[1]{\operatorname{Span}(#1)}

\renewcommand{\epsilon}{\varepsilon}
\renewcommand{\tilde}{\widetilde}

\renewcommand{\le}{\leqslant}
\renewcommand{\ge}{\geqslant}
\nc{\coloneq}{\mathrel{\mathop:}\mkern-1.2mu=}
\nc{\margin}[1]{\marginpar{\scriptsize #1}}
\nc{\para}[1]{\medskip\noindent\textbf{#1.}}
\nc{\red}[1]{\textcolor{red}{#1}}

\usepackage{listings}

\title{Linear-central filtrations and the image of the Burau representation}

\author{Nick Salter}
\email{nks@math.columbia.edu}
\date{March 26, 2019}

\address{Department of Mathematics, Columbia University, 2990 Broadway, New York, NY 10027}

\begin{document}
\maketitle

\begin{abstract} 
The Burau representation is a fundamental bridge between the braid group and diverse other topics in mathematics. A 1974 question of Birman asks for a description of the image; in this paper we give a ``strong approximation'' to the answer. Since a 1984 paper of Squier it has been known that the Burau representation preserves a certain Hermitian form. We show that the Burau image is dense in this unitary group relative to a topology induced by a naturally-occurring filtration. We expect that the methods of the paper should extend to many other representations of the braid group and perhaps ultimately inform the study of knot and link polynomials. 
\end{abstract} 

\section{introduction}
The (unreduced) Burau representation of the braid group $B_n$ is a homomorphism
\[
\beta: B_n \to \GL_n(\Lambda),
\]
where $\Lambda = \Z[t,t^{-1}]$ is the ring of integral Laurent polynomials (c.f. Section \ref{section:buraubasics}). This is a fundamental object that connects the theory of braids to fields as diverse as low-dimensional topology, knot theory, dynamics, geometric group theory, arithmetic groups, and algebraic geometry. See the references \cite{long, burau, mcmullen, churchfarb, venky, BMP} for just a few representative examples of the appearance of the Burau representation across mathematics. 

Historically much effort was devoted to the question of whether $\beta$ is injective, culminating in the work of Moody, Long--Paton, and Bigelow \cite{moody, LP, bigelow}. We consider here the complimentary question of determining the {\em image} of $\beta$; this is posed in Birman's seminal book \cite{birman} as Question 14. 

In \cite{squier}, Squier shows that the Burau representation is unitary with respect to a certain Hermitian form $J$.  Let $\Gamma$ denote the corresponding unitary group (c.f. Definition \ref{definition:gamma}). To formulate our result, we equip $\Gamma$ with a natural ``$s$-adic topology'' (c.f. Section \ref{section:sadic}, see also immediately below).

\begin{maintheorem}\label{theorem:image}
For $n\ge 5$, $\beta(B_n)$ is dense in $\Gamma$ with respect to the $s$-adic topology.
\end{maintheorem}

The $s$-adic topology is induced by the ``$s$-adic filtration'' $\{\Gamma[s^k]\}$ on $\Gamma$; since this notion lies at the center of our work, we briefly describe it here. The starting point is the fact (c.f. Lemma \ref{lemma:pure=s}) that the {\em pure} braid group $PB_n$ coincides with the kernel of the reduction map $\Gamma \to \GL_n(\Z)$ setting $t = 1$; accordingly we define $s = t-1$. Considering the {\em family} of reduction maps induced by the quotients $\Lambda \to \Lambda/ \pair{s^k}$ gives rise to the $s$-adic filtration and the $s$-adic topology. Roughly speaking, then, Theorem \ref{theorem:image} states that any $\gamma \in \Gamma$ can be approximated arbitrarily-well by the Burau representation, up to an error consisting of a matrix that is trivial modulo a large power of $s$.

The unreduced Burau representation fixes a vector $v$, leading to the definition of the {\em reduced} Burau representation 
\[
\beta': B_n \to \Gamma';
\]
here $\Gamma'$ is the restriction of $\Gamma$ to the $J$-orthogonal complement of $v$ (again see Definition \ref{definition:gamma}). It is an immediate corollary of Theorem \ref{theorem:image} that $\im(\beta')$ has the same density property as $\im(\beta)$. 

\begin{maincorollary}\label{corollary:reduced}
For $n\ge 5$, $\beta'(B_n)$ is dense in $\Gamma'$ with respect to the $s$-adic topology.
\end{maincorollary}

The question remains as to whether $\beta(B_n)$ is ``large'' within $\Gamma$ for other notions of the term. Despite our best efforts, we have been unable to find even a single element of $\Gamma$ that does not lie in the image of $\beta$. We therefore offer the following conjecture.
\begin{conjecture}
For $n \ge 5$, there are equalities
\[
\beta(B_n) = \Gamma
\]
and
\[
\beta'(B_n) = \Gamma'.
\]
\end{conjecture}

We wish to draw attention to some questions that we believe the ideas of the paper could address. 

\para{The ``Alexander spectrum''..} Since Burau's original investigations \cite{burau}, it has been known that the Burau representation is intimately connected with the Alexander polynomial of the associated braid closure. Seifert \cite{seifert} established that any polynomial satisfying two properties known to hold for the Alexander polynomial does arise as the Alexander polynomial of some knot. We call the set of possible values for the Alexander polynomial the ``Alexander spectrum''. Given the close connection between Alexander and Burau, we ask whether Theorem \ref{theorem:image} could be used to supply a quick determination of the Alexander spectrum.

\para{...and the ``Jones spectrum'' (and beyond)} This line of questioning admits a natural extension to a vast class of representations of the braid group. Many such representations are valued in a (multivariate) Laurent polynomial ring, and possess the same structure of a ``linear-central filtration'' (c.f. Section \ref{section:sadic}) that lies at the heart of the present paper. We believe it is a worthwhile endeavor to determine the image of other representations of the braid group, especially the Jones representation that gives rise to the Jones polynomial. A sufficiently precise answer would be able to determine the ``Jones spectrum'' (i.e. the set of polynomials that arise at the Jones polynomial of some knot or link), and could potentially shed new light on the structure and topological meaning of the Jones polynomial.

\para{The work of McMullen and Venkataramana} In \cite{mcmullen}, McMullen observes that the specialization of the Burau representation at $t = \zeta_d$ (a $d^{th}$ root of unity) is the monodromy representation for the family of $d$-fold cyclic branched covers of the disk. Motivated by the ``monodromy arithmeticity problem'' of Griffiths--Schmid \cite[page 123]{GS}, McMullen posed the problem of determining for which pairs $(n,d)$ is the specialization of $\beta(B_n)$ at $t = \zeta_d$ an arithmetic group. In \cite{venky}, Venkataramana shows that $\beta(B_n) \pmod{t = \zeta_d}$ is arithmetic whenever $n \ge 2d$. We ask whether Theorem \ref{theorem:image} can be used to further probe the image of Burau at roots of unity beyond the range $n \ge 2d$ required by Venkataramana. 

\para{Application to the braid Torelli group} In a companion paper \cite{kordeksalter}, Kordek and the author use the techniques developed here to study the so-called {\em braid Torelli group} $\mathcal{BI}_n$. This group is defined as the kernel of $\beta \pmod{t = -1}$.  Hain has observed \cite{hain} that $\mathcal{BI}_n$ plays an important role in the study of the {\em period map} in the theory of algebraic curves, arising as the fundamental group of the ``branch locus of the period map''. Brendle--Margalit--Putman resolved the so-called ``Hain conjecture'' in \cite{BMP}, finding a simple infinite generating set for $\mathcal{BI}_n$ and yielding some topological information about the branch locus as a result. Despite this major advance, there is a very basic question about $\mathcal{BI}_n$ that is still unknown: is $\mathcal{BI}_n$ finitely generated? In \cite{kordeksalter}, the authors use the ideas of the present paper to find a new (finite-rank) abelian quotient of $\mathcal{BI}_n$. While this does not directly resolve the question of finite generation, it offers further evidence that $\mathcal{BI}_n$ is a rich and complicated object, and highlights the role that linear-central filtrations on the Burau representation can play in this area. 

\para{Summary of the argument} In brief, Theorem \ref{theorem:image} is accomplished by first recasting the density assertion as an equality $\mathfrak g(\Gamma) = \mathfrak g(\beta)$ of certain $\Z$-Lie algebras (Proposition \ref{proposition:density}), and then showing this equality. The  $\Z$-Lie algebra $\mathfrak g(\Gamma)$ is obtained from the $s$-adic filtration $\{\Gamma[s^k]\}$ on $\Gamma$ by considering the direct sum of successive quotients $\Gamma[s^k]/\Gamma[s^{k+1}]$ (see Definition \ref{definition:g}). $\mathfrak g(\beta)$ is then the subalgebra generated by the image of $\beta$. It is easy to see that the degree-$1$ summands of $\mathfrak g(\Gamma)$ and $\mathfrak g(\beta)$ are equal (c.f. Lemma \ref{lemma:basecase}), and so one might hope that $\mathfrak g(\Gamma)$ is generated as a $\Z$-Lie algebra in degree $1$ and so prove Theorem \ref{theorem:image} essentially ``for free''.

Unfortunately, this is only true rationally (c.f. Lemma \ref{lemma:brackets}), and so one must work harder to show the equality integrally. We accomplish this by induction. Our method is to find elements of $B_n$ that are ``unexpectedly deep'' in the $s$-adic filtration and develop a technique (c.f. equation \eqref{equation:2k+1} and Lemma \ref{lemma:main}) for determining their image in $\mathfrak g(\Gamma)$. We find the first such element $\delta$ in degree $5$ (c.f. \eqref{equation:delta}) by a computer search. The induction step is accomplished by seeing that the formula of Lemma \ref{lemma:main} is ``independent of degree modulo the Lie bracket'' (c.f. Lemma \ref{lemma:commutes}). This step of the argument is the most technical and requires a delicate analysis of the expansion of matrices as power series in $s$.

\para{Outline} Section \ref{section:buraubasics} describes the Burau representation, the invariant Hermitian form, and the target group $\Gamma$ determining the codomain of $\beta$. Section \ref{section:sadic} defines the $s$-adic filtration, the associated Lie algebras $\mathfrak g(\Gamma)$ and $\mathfrak g(\beta)$, and shows that each of these are subalgebras of a more explicit $\Z$-Lie algebra $\mathfrak g$ (Lemma \ref{lemma:gembed}). The other major result of Section \ref{section:sadic} is the brief Proposition \ref{proposition:density} which reduces Theorem \ref{theorem:image} to establishing the equality $\mathfrak g(\Gamma) = \mathfrak g(\beta)$. Section \ref{section:brackets} studies the Lie bracket structure of $\mathfrak g$. Section \ref{section:mainlemma} establishes the key lemmas \ref{lemma:main} and \ref{lemma:commutes}, which allow us to probe deeper into the structure of $\mathfrak g(\beta)$. Finally, the proof of Theorem \ref{theorem:image} is carried out in Section \ref{section:proof}. 

\para{Acknowledgements} I would like to thank Kevin Kordek for many useful discussions and for collaboration on some early results in the direction of Theorem \ref{theorem:image}. I am also indebted to Alex Suciu for pointing out the connection between the $s$-adic filtration and the augmentation ideal of $\Lambda$.

\section{The Burau representation}\label{section:buraubasics}
Here we recall the basic properties of the Burau representation. Following Squier \cite{squier}, we emphasize the presence of an invariant Hermitian form, although the particular form $J$ we employ is slightly different. The key results are the classical Lemma \ref{lemma:pure=s}, which characterizes $PB_n$ in the context of the Burau representation, and Lemma \ref{lemma:unitary}, which establishes the first properties of the form $J$.

Before proceeding, we pause briefly to fix notation for the braid group $B_n$ and its subgroup of pure braids $PB_n$. For more background on these groups, see e.g. \cite[Chapter 9]{FM}. The braid group on $n$ strands has a standard generating set 
\[
B_n = \pair{\sigma_1, \dots, \sigma_{n-1}}
\]
subject to the familiar braid and commutation relations; the effect of $\sigma_i$ is to exchange the $i^{th}$ and $(i+1)^{st}$ strands. The pure braid group $PB_n$ is generated by the ${n \choose 2}$ elements $A_{ij}$  for $1 \le i < j \le n$. The effect of $A_{ij}$ is to perform a full Dehn twist about a certain circle enclosing only strands $i,j$, leaving the other strands fixed.

\subsection{Construction}
There are several complimentary approaches to the construction of $\beta$ and $\beta'$. For our purposes we will not need to know more than the explicit matrices $\{\beta(\sigma_i)\}_{i = 1}^{n-1}$; see \cite[Section 3.2, example 3]{birman} for a construction via Fox calculus, and, e.g. \cite{bigelow} for a construction via covering space theory. There are competing conventions for $\beta(\sigma_i)$ that differ by a transpose and/or an exchange of $t$ for $t^{-1}$; we opt for the following:
\begin{definition}[Unreduced Burau representation]
$\beta: B_n \to \GL_n(\Lambda)$ is the homomorphism defined by setting
\begin{equation}\label{equation:betadef}
\beta(\sigma_i) = I_{i-1} \oplus \left(\begin{array}{cc} 1-t	& 1\\ t & 0  \end{array} \right) \oplus I_{n-i-1}.
\end{equation}
\end{definition}
It is immediate from this that $\beta$ fixes the vector 
\begin{equation}\label{equation:v}
v = (t,t^2, \dots, t^n)^T
\end{equation}
(acting by left multiplication), and that also $\beta$ fixes the vector
\begin{equation}\label{equation:1}
\vec 1 = (1, 1, \dots, 1)
\end{equation}
when acting by right multiplication (in later portions of the argument, we will also wish to consider the {\em column} vector consisting of all $1$'s; this will also be denoted $\vec 1$). Taking a $\beta$-invariant summand $W \le \Lambda^n$ complimentary to $\Lambda \pair{v}$, we obtain the {\em reduced Burau representation}
\[
\beta': B_n \to \GL(W) \cong \GL_{n-1}(\Lambda).
\]

One of the fundamental properties of $\beta$ that we exploit is the fact that the pure braid group admits a natural definition in terms of the algebraic structure of $\GL_n(\Lambda)$. 

\begin{lemma}\label{lemma:pure=s}
Let $r: \GL_n(\Lambda) \to \GL_n(\Z)$ denote the homomorphism induced from the ring morphism $\Lambda \to \Z$ obtained by setting $t = 1$. Then $r \circ \beta: B_n \to \GL_n(\Z)$ coincides with the permutation representation of the symmetric group. Consequently there is an equality
\[
PB_n = \ker(r \circ \beta).
\]
\end{lemma}
\begin{proof}
Setting $t = 1$ in \eqref{equation:betadef} sends $\sigma_i$ to the permutation representation of the transposition $(i\ i+1)$. Thus setting $t = 1$ for an arbitrary element $x \in B_n$ simply records the associated permutation $\bar x \in S_n$, and the result follows. 
\end{proof}

\subsection{Unitarity}
In the 1984 paper \cite{squier}, Squier made a fundamental contribution to the problem of determining the image of the Burau representation by identifying a certain invariant Hermitian form for $\beta'$. We do not follow Squier's original definition here, as we will need to make use of the unitarity of the {\em unreduced} Burau representation (Squier treated only the reduced case). The unitarity of $\beta'$ is {\em a posteriori} very natural, and follows from the covering-space construction of $\beta'$ as automorphisms on a certain homology group $H_1(\tilde D_n)$ that is equipped with a $\Lambda$-valued intersection pairing. The unitarity of $\beta$ is still surprising from this point of view since the corresponding topological construction involves a {\em relative} homology group for which no such pairing seems to exist. However, one can easily extend Squier's form to a Hermitian form on the unreduced Burau representation by the requirement that the invariant subspace be orthogonal. To avoid the presence of some superfluous roots of $t$ involved in Squier's form, we adopt a different approach. The form $J$ below was obtained directly by trial-and-error (essentially by reverse-engineering the proof of Lemma \ref{lemma:unitary})! 

\begin{definition} [the Hermitian form $J$] \label{definition:J}
$J \in \GL_n(\Lambda)$ is the matrix with entries
\[
(J)_{ij} = \begin{cases}
		1 	& \mbox{ if } i = j\\
		-t	& \mbox{ if } i >j\\
		-t^{-1}& \mbox{ if } i <j.
\end{cases}
\]
\end{definition} 

\begin{lemma}\label{lemma:unitary}
Let $J$ be as in Definition \ref{definition:J}. The unreduced Burau representation $\beta$ is Hermitian with respect to $J$ (relative to the involution $\bar\cdot: \Lambda \to \Lambda$ determined by $\bar t = t^{-1}$).
\end{lemma}
\begin{proof}
Define the involution $\cdot^*: M_n(\Lambda) \to M_n(\Lambda)$ by
\begin{equation}\label{equation:*}
(A^*)_{ij} = \bar{(A)_{ji}}.
\end{equation}
We must show that 
\[
(\beta(\sigma_i))^* J \beta(\sigma_i) = J
\]
for $1 \le i \le n-1$. This is a routine calculation which we omit.
\end{proof}

\subsection{The target group $\Gamma$}\label{subsection:gamma}
In \eqref{equation:v} we saw that $\im(\beta)$ fixes the vector $v$, and in Lemma \ref{lemma:unitary}, we saw that $\im(\beta)$ preserves the Hermitian form $J$. One is led to wonder how close these conditions come to characterizing the image of $\beta$. There is one final ``obvious'' constraint on the image of $\beta$ that follows from Lemma \ref{lemma:pure=s}: the image of $\beta$ mod $t-1$ must be the permutation representation of $S_n$. The reduction of $J$ mod $t-1$ is a positive-definite Hermitian form, and hence its isometry group is finite, but it is in fact a {\em proper} extension of $S_n$ and so this must be taken into account when characterizing the image of $\beta$. Theorem \ref{theorem:image} asserts that from the point of view of the $s$-adic topology, there are no other constraints on the image of $\beta$.

\begin{definition}[The groups $\Gamma, \Gamma'$]\label{definition:gamma}
Let $v$ be the globally-fixed vector defined in \eqref{equation:v}, let $J$ be the Hermitian form of Definition \ref{definition:J}, and let $(S_n)_{perm}$ denote the embedding of $S_n$ into $\GL_n(\Z)$ as the permutation representation.
Then the group $\Gamma$ is defined as follows:
\[
\Gamma = \{A \in \GL_n(\Lambda) \mid A v= v,\ A^* J A = J,\ A \pmod{t-1} \in (S_n)_{perm}\}.
\]
Let $\Gamma' \leqslant \GL_{n-1}(\Lambda)$ denote the image of $\Gamma$ when restricted to the $J$-orthogonal complement $v^\perp$. \
\end{definition}

\section{The $s$-adic filtration on $\Gamma$ and the Lie algebra $\mathfrak g$}\label{section:sadic}
In this section we begin our study of the algebraic structure of $\Gamma$. Inspired by the work of Lee--Szczarba \cite{LS}, as well as the notion of a ``linear-central filtration'' appearing in the work of Bass--Lubotzky \cite{BL}, we introduce the notions of an ``$s$-adic filtration'' on $\Gamma$ and a trio of closely-related graded $\Z$-Lie algebras $\mathfrak g, \mathfrak g(\Gamma), \mathfrak g(\beta)$; these appear as Definitions \ref{definition:sadic} and \ref{definition:g}, respectively. In Lemma \ref{lemma:gembed} we observe that there are natural containments $\mathfrak g(\beta) \subseteq \mathfrak g(\Gamma) \subseteq \mathfrak g$. Finally in Proposition \ref{proposition:density}, we reduce Theorem \ref{theorem:image} to the problem of showing the equality $\mathfrak g(\Gamma) = \mathfrak g(\beta)$. 

\subsection{The $s$-adic filtration}
Lemma \ref{lemma:pure=s} shows that $PB_n$ admits a characterization in terms of the algebraic structure of $\Gamma$. It also highlights the importance of the reduction setting $t = 1$. The key technical idea of the paper is to push this insight further to probe deeper and deeper into the structure of $B_n$ and $\Gamma$. Set
\begin{equation}\label{equation:s}
s = t-1.
\end{equation}

\begin{definition} [$s$-adic filtration] \label{definition:sadic}
The {\em $s$-adic filtration} on $\Gamma$ is the decreasing filtration $\{\Gamma[s^k]\}_{k \ge 0}$ where $\Gamma[s^k]$ is defined as the kernel of the reduction map 
\[
\Gamma \to \GL_n(\Lambda / \pair{s^k}).
\]
We also define the $s$-adic filtration on $B_n$ by pullback:
\[
B_n[s^k]:= B_n \cap \beta^{-1}(\Gamma[s^k]).
\]
\end{definition} 

\begin{lemma}\label{lemma:exhaustive}
The $s$-adic filtration on $\Gamma$ is exhaustive:
\[
\bigcap_{k \ge 1} \Gamma[s^k] = \{I\}
\]
\end{lemma}

\begin{proof}
For $A \neq I \in \Gamma$, consider the matrix $A - I$. By definition, $A \in \Gamma[s^k]$ if and only if $s^k$ divides each entry of $A-I$. Only $0 \in\Lambda$ is divisible by $s^k$ for all $k \ge 0$, from which the result follows.
\end{proof}

\begin{remark}
As pointed out to us by Alex Suciu, for the choice $s = t-1$, the $s$-adic filtration on $\Lambda$ coincides with the filtration on $\Lambda$ by powers of the augmentation ideal (viewing $\Lambda$ as the group ring of $\Z$). We do not emphasize this point of view further here, since we believe this to be largely coincidental. In \cite{kordeksalter}, we will study a different ``$s'$-adic filtration'' on $\Gamma$ taking $s' = t+1$.
\end{remark}

We next develop the crucial idea of ``$s$-adic expansion''.
\begin{definition}[$s$-adic expansion]\label{definition:expansion}
Let $A \in \Gamma$ be given. The {\em $s$-adic expansion} of $A$ is the formal power series in $M_n(\Z)[[s]]$
\[
A = (A)_{(0)} + s (A)_{(1)} + s^2 (A)_{(2)} + \dots,
\]
with $(A)_{(0)} = A \pmod s$ and each $(A)_{(k)}$ for $k \ge 1$ defined recursively by the formula
\[
(A)_{(k)} = (A - ((A)_{(0)} + \dots + s^{k-1}(A)_{(k-1)}))/(s^k) \pmod s.
\]
For $k \ge 0$, the {\em truncated $s$-adic expansion} $(A)_{(\le k)}$ is the finite sum
\[
\sum_{i = 0}^k s^i (A)_{(i)} \in M_n(\Lambda).
\]
\end{definition}
We will also make use of the $s$-adic expansion of vectors $v \in \Lambda^n$ using the same notation. The following lemma is immediate.
\begin{lemma}\label{lemma:reduction}
For any $k \ge 0$ and any $A \in \Gamma$, the image of $A$ in $\GL_n(\Lambda/\pair{s^{k+1}})$ is given by the truncated $s$-adic expansion $(A)_{(\le k)}$. 
\end{lemma}

The following lemma will be employed in the proof of Lemma \ref{lemma:bracketgrading} below.
\begin{lemma}\label{lemma:sadicinverse}
For $k \ge 1$ and $X \in \Gamma[s^k]$ arbitrary, the truncated  $s$-adic expansion of $X^{-1}$ is given as follows:
\[
(X^{-1})_{(\le 2k-1)} = I - s^k(X)_{(k)} - s^{k+1} (X)_{(k+1)} - \dots - s^{2k-1} (X)_{(2k-1)}.
\]
\end{lemma}
\begin{proof}
Let $A, B \in \Gamma$ be arbitrary. The $j^{th}$ term of the $s$-adic expansion of the product $AB$ is given as follows:
\begin{equation}\label{equation:prodformula}
(AB)_{(j)} = \sum_{i = 0}^j (A)_{(i)} (B)_{(j-i)}.
\end{equation}
For $A = X, B= X^{-1}$, the terms $(X X^{-1})_{(j)} = (I)_{(j)}$ vanish for $j \ge 1$. By hypothesis, $(X)_{(j)} = 0$ for $1 \le j < k$, and the same holds for $X^{-1}$ since $X^{-1} \in \Gamma[s^k]$ as well. Examining \eqref{equation:prodformula} for $j = k$ shows that $(X^{-1})_{(k)} = -(X)_{(k)}$, and one proceeds like this, inductively showing that the result holds for all coefficients $(X^{-1})_{(j)}$ for $j < 2k$.
\end{proof}

In Section \ref{subsection:lie}, we will see how the $s$-adic filtration on $\Gamma$ gives rise to some closely associated graded $\Z$-Lie algebras. In preparation for this construction, we establish the following two lemmas.

\begin{lemma}\label{lemma:bracketgrading}
For $1 \le k \le \ell$, the commutator is compatible with the $s$-adic filtration: given $X \in \Gamma[s^k]$ and $Y \in \Gamma[s^\ell]$,
\[
[X,Y] \in \Gamma[s^{k + \ell}].
\]
Moreover,
\begin{equation}\label{equation:commbrak}
([X,Y])_{(k + \ell)} = (X)_{(k)}(Y)_{(\ell)} - (Y)_{(\ell)}(X)_{(k)} := \pair{(X)_{(k)},(Y)_{(\ell)}} .
\end{equation}
\end{lemma}
\begin{proof}
Write
\[
XY = X(I + s^\ell (Y)_{(\ell)} + \dots) = X + s^\ell X (Y)_{(\ell)} + \dots
\]
and likewise,
\[
X^{-1}Y^{-1} = X^{-1}(I + s^\ell (Y^{-1})_{(\ell)}+\dots) = X^{-1} + s^\ell X^{-1}(Y^{-1})_{(\ell)} + \dots
\]
Multiplying,
\begin{equation}\label{equation:XYXY}
(XY)(X^{-1}Y^{-1})  = I + s^\ell (X (Y)_{(\ell)} X^{-1} + (Y^{-1})_{(\ell)}) + \dots
\end{equation}
For $j \ge \ell$ arbitrary, the expression $s^j(X (Y)_{(j)}X^{-1})$ has terms in degree $j, j+k, j+k+1, \dots$, and the degree-$j$ term is given by $s^j (Y)_{(j)}$. Collecting terms of like degree in \eqref{equation:XYXY}, we see that for $\ell \le j \le k + \ell -1$, the $s^j$ term is given by
\[
(Y)_{(j)} + (Y^{-1})_{(j)}.
\]
As we are assuming $\ell \ge k$, it follows from Lemma \ref{lemma:sadicinverse} that in this range, $(Y^{-1})_{(j)} = -(Y)_{(j)}$. Consequently the $s^j$-term of the $s$-adic expansion of $[X,Y]$ vanishes for $1 \le j \le k + \ell - 1$, so that $[X,Y] \in \Gamma[s^{k+\ell}]$ as claimed. The computation of $([X,Y])_{(k + \ell)}$ follows directly by the same analysis.
\end{proof}

\begin{remark}[The $s$-adic filtration is linear-central]\label{remark:LC}
In \cite{BL}, Bass--Lubotzky define a ``linear-central filtration'' on a group $\Gamma$ to be a filtration $\{\Gamma_i\}_{i \ge 0}$ such that (a) $\Gamma_0 / \Gamma_1$ is a linear group and (b) $[\Gamma_i, \Gamma_j] \le \Gamma_{i+j}$. We observe here that the $s$-adic filtration is linear-central: (a) follows from Lemma \ref{lemma:pure=s}, and (b) was established immediately above in Lemma \ref{lemma:bracketgrading}. 
\end{remark}

\begin{lemma}\label{lemma:abquotient}
For $k \ge 1$, the quotient $\Gamma[s^k]/\Gamma[s^{2k}]$ is abelian; {\em a fortiori} the quotient $\Gamma[s^k]/\Gamma[s^{k+1}]$ is abelian as well. The homomorphism $(\cdot)_{(k)}: \Gamma[s^k] \to M_n(\Z)$ descends to an embedding of $\Gamma[s^k]/\Gamma[s^{k+1}]$ into $M_n(\Z)$.
\end{lemma}
\begin{proof}
The first claim follows immediately from Lemma \ref{lemma:bracketgrading}: 
\[
[\Gamma[s^k],\Gamma[s^k]] \le \Gamma[s^{2k}] \le \Gamma[s^{k+1}].
\]
For the second claim, if $A \in \Gamma[s^k]$ with $(A)_{(k)} = 0$, then $A \in \Gamma[s^{k+1}]$ by definition, from which the second claim follows. 
\end{proof}

\subsection{Some $\Z$-Lie algebras}\label{subsection:lie}
Groups equipped with a commutator-compatible filtration as in Lemma \ref{lemma:bracketgrading} have a naturally-associated graded $\Z$-Lie algebra. Here we define this Lie algebra for the $s$-adic filtration on $\Gamma$ and establish some first properties as summarized in Lemma \ref{lemma:gembed}. From there we turn to Proposition \ref{proposition:density}, which gives a reformulation of Theorem \ref{theorem:image} in the setting of $\Z$-Lie algebras. 

\begin{definition}[the Lie algebras $\mathfrak g, \mathfrak g(\Gamma), \mathfrak g(\beta)$]\label{definition:g}\ 
\begin{enumerate}
\item $\mathfrak g$ is the graded $\Z$-Lie algebra 
\[
\mathfrak g := \bigoplus_{k \ge 1} G_k,
\]
with $G_k$ defined as follows:
\[
G_k = \begin{cases}
		\{M \in M_n(\Z) \mid M\vec1 = 0,\  M^T = M\}			& k = 1\\
		\{M \in M_n(\Z) \mid M\vec 1 = 0,\ M^T = -M\}			& k \ge 2 \mbox{ even}\\
		\{M \in M_n(\Z) \mid M\vec 1 = 0,\ M^T = M,\ \tr(M) = 0\} 	& k \ge 3 \mbox{ odd}
\end{cases}
\]
The Lie bracket on $\mathfrak g$ is induced from the usual bracket $\pair{A,B} = AB- BA$ on $M_n(\Z)$. 

\item $\mathfrak g(\Gamma)$ is the graded $\Z$-Lie algebra
\[
\mathfrak g(\Gamma) := \bigoplus_{k \ge 1} \Gamma[s^k]/\Gamma[s^{k+1}]
\]
with Lie bracket induced by the commutator $[\cdot, \cdot]: \Gamma[s^k] \times \Gamma[s^{\ell}] \to \Gamma[s^{k+\ell}]$ (that is this in fact a Lie bracket follows from Lemma \ref{lemma:bracketgrading}).

\item $\mathfrak g(\beta)$ is the subalgebra of $\mathfrak g(\Gamma)$ determined by the image of $\beta$:
\[
(\mathfrak g(\beta))_k := \beta(B_n \cap \Gamma[s^k]) \pmod{\Gamma[s^{k+1}]}.
\]

\end{enumerate}
\end{definition}

Of these Lie algebras, $\mathfrak g(\Gamma)$ is the one ``naturally associated'' to the $s$-adic filtration on $\Gamma$. Ultimately we will see that these are all isomorphic, but this is essentially equivalent to Theorem \ref{theorem:image}; see Proposition \ref{proposition:density}. As a first step in this direction, we see that $\mathfrak g(\Gamma)$ embeds in $\mathfrak g$.

\begin{lemma}\label{lemma:gembed}
For $k \ge 1$, each quotient $\Gamma[s^k]/\Gamma[s^{k+1}]$ is isomorphic to a subgroup of $G_k$, and the induced embedding of graded abelian groups $\mathfrak g(\Gamma) \subseteq \mathfrak g$ is in fact an embedding of $\Z$-Lie algebras.
\end{lemma}
\begin{proof}
Lemma \ref{lemma:abquotient} identifies $\Gamma[s^k]/\Gamma[s^{k+1}]$ with a subgroup of $M_n(\Z)$. We must show that the conditions characterizing $G_k$ in Definition \ref{definition:g} hold for any $A \in \Gamma[s^k]/\Gamma[s^{k+1}]$. We must furthermore show that under the embedding $\Gamma[s^k]/\Gamma[s^{k+1}] \le M_n(\Z)$, the commutator in $\Gamma$ is represented by the matrix bracket. Each of these will require a separate argument. The bracket has already been established in \eqref{equation:commbrak} in Lemma \ref{lemma:bracketgrading}; the remaining three points are listed below as Lemmas \ref{lemma:step1}-\ref{lemma:step3}. 

\begin{lemma}\label{lemma:step1}
For $k \ge 1$, the image of $A \in \Gamma[s^k]$ in $M_n(\Z)$ satisfies $A \vec 1 = 0$.
\end{lemma}
\begin{proof}
Recall the globally-fixed vector $v \in \Lambda^n$ of \eqref{equation:v}. Since $v$ is fixed, for any $A \in \Gamma[s^k]$, 
\[
(A-I)v = 0.
\]
The $s$-adic expansion of $v = (t, t^2, \dots, t^n)^T$ is given by
\[
v = \vec 1 + s (v)_{(1)} + \dots
\]
Taking the $s$-adic expansion of the product $(A-I)v = 0$, we find that the $s^k$ term is given by the product $(A)_{(k)} \vec 1$, from which the claim follows. 
\end{proof}

\begin{lemma}\label{lemma:step2}
For $k \ge 1$, the image of $A \in \Gamma[s^k]$ in $M_n(\Z)$ satisfies $A^T = (-1)^{k+1} A$.
\end{lemma}
\begin{proof}
The (skew-)symmetry condition will follow from the Hermitian structure on $\Gamma$ given by $J$. We will examine the $s^k$ term in the $s$-adic expansion for $J = A^* J A$. To do this, we must compare the $s$-adic expansions of $A$ and $A^*$. As 
\[
s^* = (t-1)^* = (t^{-1}-1) = \frac{-s}{1+s},
\]
we see that $(A^*)_{(k)} = (-1)^k (A)_{(k)}^T$. Extracting the $s^k$ term in the expansion for $(A^* J A - J) = 0$, we find
\[
(J)_{(0)} (A)_{(k)} + (-1)^k (A)_{(k)}^T (J)_{(0)} = 0.
\]
Let $\bold 1$ denote the $n \times n$ matrix with every entry equal to $1$. Setting $t = 1$ in Definition \ref{definition:J}, 
\begin{equation}\label{equation:J0}
(J)_{(0)} = 2 I - \bold 1.
\end{equation} 
As $\vec 1 A = \vec 1$ by construction, a look at the $s$-adic expansions shows that $\vec 1 (A)_{(j)} = 0$ for all $j \ge 1$. Thus 
\[
(J)_{(0)} (A)_{(k)} + (-1)^k (A)_{(k)}^T (J)_{(0)} = 2 ((A)_{(k)} + (-1)^k (A)_{(k)}^T) = 0,
\]
from which the (skew-) symmetry follows.
\end{proof}

\begin{lemma}\label{lemma:step3}
For $k \ge 2$, the image of $A \in \Gamma[s^k]$ in $M_n(\Z)$ satisfies $\tr((A)_{(k)}) = 0$.
\end{lemma}
\begin{proof}
Following Lemma \ref{lemma:step2}, this property holds automatically for $k$ even, but our argument will not need to specialize to $k$ odd. To establish the claim, we will examine the characteristic polynomial of $A-I$. Evaluating,
\[
\chi_{A-I}(-1) = \chi_A(0) = (-1)^n \det(A) = \pm t^a
\]
for some $a \in \Z$, since $A \in \GL_n(\Lambda)$ and $\Lambda^\times  = \{\pm t^a\mid a \in \Z\}$. By assumption, $A-I = s^k A'$ for some $A' \in M_n(\Lambda)$. The claim $\tr((A)_{(k)}) = 0$ will follow if we can show that $\tr(A')$ is divisible by $s$; this is what we will prove.

Computing the characteristic polynomial,
\[
\chi_{A-I}(x) := \det(xI - s^kA') = s^{nk} \chi_{A'}(x/(s^k)).
\]
It follows that for $0 \le j \le n$, the coefficient on $x^j$ is divisible by $s^{(n-j)k}$. We can therefore write
\begin{equation}\label{equation:charpoly}
\chi_{A-I}(-1) = \pm t^a = (-1)^n (1- \tr(A') s^k + s^{2k} \nu)
\end{equation}
for some $\nu \in \Lambda$. Reducing this expression mod $s$, we find
\[
\pm 1 = (-1)^n, 
\]
which shows that $\det(A) = (-1)^n t^a$ for some $a \in \Z$. Returning to \eqref{equation:charpoly}, we find
\[
t^a-1 = s (1 + t + \dots + t^{a-1})= -s^k( \tr(A') + s^{k} \nu)
\]
if $a > 0$ and a similar expression with $t^{-1}$ in place of $t$ if $a <0$. Dividing by $s$ and then reducing mod $s$ we reach a contradiction (as $k \ge 2$). We conclude that \eqref{equation:charpoly} can only hold for $a = 0$.

To summarize, we have established that $\det(A) = 1$. Returning to \eqref{equation:charpoly} one last time, we see that
\[
s^k \tr(A') = s^{2k} \nu,
\]
so that $\tr(A')$ is divisible by $s^k$ as desired. 
\end{proof}
This concludes the proof of Lemma \ref{lemma:gembed}.
\end{proof}

\begin{remark}\label{remark:whyreduced}
Our choice to work with the {\em unreduced} Burau representation is largely motivated by Lemma \ref{lemma:gembed}. There are corresponding $\Z$-Lie algebras $\mathfrak g', \mathfrak g'(\Gamma')$ attached to the codomain of the reduced Burau representation, but the corresponding groups $G_k'$ are defined by much less natural equations than the (skew-)symmetry appearing in Lemma \ref{lemma:step2}.
\end{remark}

\para{Density via Lie algebras} Our final objective in this section is to reformulate the density result claimed in Theorem \ref{theorem:image} in terms of the $\Z$-Lie algebras $\mathfrak g(\Gamma)$ and $\mathfrak g(\beta)$. 

\begin{proposition}\label{proposition:density}
Suppose there is an equality
\[
\mathfrak g(\Gamma) = \mathfrak g(\beta).
\]
Then $\im(\beta)$ is dense in $\Gamma$ in the $s$-adic topology.
\end{proposition}
\begin{proof}
Let $\gamma \in \Gamma$ be arbitrary. It suffices to construct a sequence $A_0, A_1, \dots$ of elements in $\im(\beta)$ with $\gamma^{-1} A_k \in \Gamma[s^{k+1}]$. This will be constructed inductively. Let $A_0 \in \im(\beta)$ be chosen so that $(A_0)_{(0)} = (\gamma)_{(0)}$ as elements of $\Gamma / \Gamma[s]$. Assuming that $A_0, \dots, A_{k-1}$ have been constructed, it follows that
\[
\gamma^{-1} A_{k-1} \in \Gamma[s^{k}].
\]
By hypothesis, there is some $B_{k} \in \im(\beta) \cap \Gamma[s^{k}]$ such that $(B_{k})_{(k)} = \gamma^{-1} A_{k-1} \pmod{\Gamma[s^{k+1}]}$. Then $A_k$ can be constructed by setting $A_k = A_{k-1} B_k^{-1}$. 
\end{proof}

\section{Brackets in $\mathfrak g$}\label{section:brackets}
This section is devoted to a study of the $\Z$-Lie algebra $\mathfrak g$. The central result is Lemma \ref{lemma:brackets}, which studies the extent to which the Lie bracket maps $\pair{\cdot,\cdot}: G_1 \otimes G_k \to G_{k+1}$ are surjective. Prior to this we establish some basic properties of $\mathfrak g$ as an $S_n$-module (Lemma \ref{lemma:gSn}), as well as some useful bracket formulas in Lemma \ref{lemma:bracketformula}.

\subsection{$\mathfrak g(\Gamma)$ and $\mathfrak g$ as $S_n$-modules} 
In the language of Bass--Lubotzky \cite{BL}, $\{\Gamma[s^k]\}$ is a ``linear-central filtration'', and so the $\Z$-Lie algebra $\mathfrak g(\Gamma)$ carries an action of $\Gamma / \Gamma[s] \cong S_n$ induced by the conjugation action of $\Gamma$ on the normal subgroups $\Gamma[s^k]$. In this subsection we record this fact and some of its consequences.

\begin{lemma}\label{lemma:gismodule}
There is a grading-preserving action of $S_n$ on $\mathfrak g(\Gamma)$ induced from the conjugation action of $\Gamma$ on the terms $\Gamma[s^k]$ of the $s$-adic filtration. Under the embedding $\mathfrak g (\Gamma) \le \mathfrak g$ of Lemma \ref{lemma:gembed}, $S_n$ acts on each summand $G_k$ by conjugation by the standard permutation matrices in $\GL_n(\Z)$. 
\end{lemma}

\begin{proof}
It follows from Lemma \ref{lemma:bracketgrading} that for $k \ge 1$, the conjugation action of $\Gamma$ on $\Gamma[s^k]$ induces an action of $\Gamma/\Gamma[s]$ on $\Gamma[s^k]/\Gamma[s^{k+1}]$. We have already observed in Lemma \ref{lemma:pure=s} that $\Gamma / \Gamma[s]$ is a subgroup of $\GL_n(\Z)$ isomorphic to $S_n$ represented as the group of permutation matrices. 
\end{proof}

\subsection{Generating $\mathfrak g$}
Lemma \ref{lemma:gismodule} establishes that $\mathfrak g$ is a module over $S_n$. In this subsection, we develop an explicit set of generators for $\mathfrak g$ as an $S_n$-module.

\begin{definition}[$X_{ij}$ and $Y_{ijk}$]\label{definition:XY}
For $1 \le a,b \le n$, let $E_{ab} \in M_n(\Z)$ denote the matrix with entry $(E_{ab})_{ab} = 1$ and with all other entries zero. For $i \ne j$, define $X_{ij} \in G_1$ by
\[
X_{ij} = E_{ii} + E_{jj} - E_{ij} - E_{ji},
\]
and for $1\le i< j< k\le n$, define $Y_{ijk} \in G_2$ by
\[
Y_{ijk} = (E_{ij} - E_{ji}) -(E_{ik}- E_{ki}) + (E_{jk} - E_{kj}).
\]
\end{definition}

\begin{lemma}\label{lemma:gSn}\ ~
\begin{enumerate}
\item\label{item:odd} For $n \ge 5$ and $k \ge 3$ odd, $G_k$ is generated as an $S_n$-module by any element of the form $X_{ij} - X_{k\ell}$ for integers $i,j,k,\ell$ such that $\#\{i,j,k,\ell\} \ge 3$. 
\item \label{item:even} For $n \ge 3$ and $k \ge 2$ even, $G_k$ is generated as an $S_n$-module by any element of the form $Y_{ijk}$ for distinct integers $i,j,k$.\end{enumerate}
\end{lemma}

\begin{proof}
We leave it to the reader to verify that $G_1$ is generated as an abelian group by the collection of elements $X_{ij}$ for $1 \le i < j \le n$. Each such $X_{ij}$ has trace $2$, and so for $k \ge 3$ odd, $G_k$ is generated by the collection of elements $X_{ij} - X_{k\ell}$ for some possibly-redundant set of indices $\{i,j,k,l\}$. It is clear that there is exactly one $S_n$-orbit each of three- and four-element index sets, so it remains only to show how to express an element $X_{ij}- X_{ik}$ in the $\Z[S_n]$-span of $X_{ij}-X_{k\ell}$, and conversely. The first of these is shown below; the other is similar.
\[
X_{12} - X_{13} = (X_{12} - X_{45}) - (X_{13} - X_{45}).
\]

The proof of (2) is essentially immediate - it is only necessary to check that the collection of all $Y_{ijk}$ span $G_{k}$ for $k$ even, and this is a matter of elementary linear algebra over $\Z$.
\end{proof}

We next begin the process of studying $\mathfrak g$ as a $\Z$-Lie algebra. In what follows we will only need to rely on the following computations, all of which follow by direct verification.

\begin{lemma}\label{lemma:bracketformula}
We have the following relations among the elements $X_{ij}, Y_{ijk} \in \mathfrak g$:
\begin{enumerate}
\item \label{item:xx1} For $1 \le i < j < k \le n$,
\[
\pair{X_{ij}, X_{ik}}= Y_{ijk}.
\]
\item\label{item:xx2} If $\#\{i,j,k,\ell\} \ne 3$, then
\[
\pair{X_{ij},X_{k\ell}}= 0.
\]
\item \label{item:xy}
\[
\pair{X_{ij}, Y_{ijk}} = 2 (X_{ik} - X_{jk}).
\]
\end{enumerate}
\end{lemma}

It turns out $\mathfrak g$ is {\em not} generated by $G_1$ as a $\Z$-Lie algebra, but that this is true ``away from $2$''. The next lemma makes this precise.

\begin{lemma}\label{lemma:brackets}
For $n \ge 5$ and $k \ge 1$, there is an equality
\[
\pair{G_1,G_{2k-1}} = G_{2k}
\]
and a containment
\[
2 G_{2k+1} \le \pair{G_1, G_{2k}}.
\]
\end{lemma}
\begin{proof}
Following Lemma \ref{lemma:gSn}, it suffices to exhibit an element of the form $Y_{ijk}$ in the image of $\pair{\cdot, \cdot}: G_1 \otimes G_{2k-1} \to G_{2k}$, and to exhibit some $2 (X_{ik} - X_{jk})$ in the image of $\pair{\cdot, \cdot}: G_1 \otimes G_{2k} \to G_{2k+1}$. Both of these follow readily from Lemma \ref{lemma:brackets}, the second of which is actually immediate by Lemma \ref{lemma:brackets}.\ref{item:xy}. For the first, we combine items \eqref{item:xx1} and \eqref{item:xx2} of Lemma \ref{lemma:brackets} to observe that
\[
\pair{X_{12},X_{13}-X_{34}} = Y_{123}
\]
as required.
\end{proof}

\section{Beyond the associated graded: the main lemma}\label{section:mainlemma}
As observed in Proposition \ref{proposition:density}, Theorem \ref{theorem:image} will follow from the equality $\mathfrak g(\Gamma) = \mathfrak g(\beta)$. This will be established by induction on the degree of grading. In this section, we formulate and prove Lemma \ref{lemma:commutes}, which will serve as the crucial tool used to establish the inductive step. To do so, we will first establish Lemma \ref{lemma:2k+1}, which shows how to extract some information on commutators that is invisible to the graded Lie algebra $\mathfrak g$.

Lemma \ref{lemma:brackets} implies that the Lie bracket on $\mathfrak g$ induces a short exact sequence
\[
1 \to K_{2k-1} \to G_1\otimes G_{2k-1} \to G_{2k} \to 1.
\]
Said differently, if
\begin{equation}\label{equation:a}
a = \sum_{i=1}^m X_{I_i} \otimes W_i 
\end{equation}
is an arbitrary element of $K_{2k-1}$ and 
\begin{equation}\label{equation:alpha}
\alpha = \prod_{i=1}^m [\beta(A_{I_i}), \omega_i] 
\end{equation}
is any element of $\Gamma$ with each $\omega_i \in \Gamma[s^{2k-1}]$ satisfying $(\omega_i)_{(2k-1)} = W_i \in G_{2k-1}$, then $\alpha \in \Gamma[s^{2k+1}]$. We see in Lemma \ref{lemma:commutes} below that the Hermitian structure on $\Gamma$ allows one to recover some information about $(\alpha)_{(2k+1)}$ purely from the associated $a \in K_{2k-1}$. Before being able to accomplish this, we require a preliminary study of $(\alpha)_{(2k+1)}$.

\begin{lemma}\label{lemma:2k+1}
Let $\alpha \in \Gamma[s^{2k+1}]$ be given as in \eqref{equation:alpha}, with $\omega_i \in \Gamma[s^{2k-1}]$ for $1 \le i \le m$. Then
\begin{equation}\label{equation:2k+1}
(\alpha)_{(2k+1)} = \sum_{i = 1}^m \left( \pair{X_{I_i}, (\omega_i)_{(2k)}} + \pair{\beta(A_{I_i})_{(2)}, W_i}+\pair{W_i,X_{I_i}}X_{I_i}\right).
\end{equation}
\end{lemma}

\begin{proof}
This is a direct computation with $s$-adic expansions, building off of \eqref{equation:XYXY} in the proof of Lemma \ref{lemma:bracketgrading}. This expression is reproduced (in a slightly expanded form) for convenience below.
\[
[X,Y]  = I + s^\ell (X (Y)_{(\ell)} X^{-1} + (Y^{-1})_{(\ell)}) +s^{\ell+1} (X (Y)_{(\ell+1)} X^{-1} + (Y^{-1})_{(\ell+1)})	+ \dots
\]
We apply this formula with $X = \beta(A_{ij})$ and $Y = \omega \in \Gamma[s^{2k-1}]$ arbitrary. For $\ell = 2k-1,2k,2k+1$, the summand
\[
s^\ell(X (Y)_{(\ell)} X^{-1} + (Y^{-1})_{(\ell)})
\]
contributes to $([X,Y])_{(2k+1)}$. Since $2k-1 \ge 3$, Lemma \ref{lemma:sadicinverse} implies that $(Y^{-1})_{(\ell)} = -(Y)_{(\ell)}$ in this range. One also verifies that $(X^{-1})_{(2)} = (X)_{(1)}^2 - (X)_{(2)}$. Armed with these facts, the claim now follows by a direct computation.
\end{proof}

\begin{lemma}\label{lemma:main}
For $k \ge 2$ there is a homomorphism
\[
\phi_{2k-1}: K_{2k-1} \to G_{2k+1}/\pair{G_1,G_{2k}}
\]
defined as follows: for $a$ and $\alpha$ as in \eqref{equation:a}, \eqref{equation:alpha} as above, the assignment 
\[
\phi_{2k-1}(a) =(\alpha)_{(2k+1)}
\]
is a well-defined element of $G_{2k+1}/\pair{G_1,G_{2k}}$. \end{lemma}
\begin{proof}
It suffices to show that if $\alpha$ is replaced by some
\[
\alpha'= \alpha = \prod_{i=1}^m [\beta(A_{I_i}), \omega'_i],
\]
with each $(\omega'_i)_{(2k-1)} = W_i$, then 
\[
(\alpha \alpha'^{-1})_{(2k+1)} \in \pair{G_1,G_{2k}}.
\]
To see this, we consider the image of $\alpha \alpha'^{-1}$ in $\Gamma/\Gamma[s^{2k+2}]$. As $\beta(A_{I_i}) \in \Gamma[s]$ and $\omega_i, \omega'_i \in \Gamma[s^{2k-1}]$, we use Lemma \ref{lemma:bracketgrading} to write
\[
\alpha \alpha'^{-1} \equiv \prod_{i=1}^m [\beta(A_{I_i}),\omega_i (\omega'_i)^{-1}] \pmod {\Gamma[s^{2k+2}]}.
\]
By hypothesis, each $\omega_i (\omega'_i)^{-1} \in \Gamma[s^{2k}]$, so that each summand 
\[
([\beta(A_{I_i}),\omega_i (\omega'_i)^{-1}])_{(2k+1)} = \pair{X_{I_i},(\omega_i (\omega'_i)^{-1})_{(2k)}}
\]
 is an element of $\pair{G_1, G_{2k}}$ as required.
\end{proof}

\begin{lemma}\label{lemma:commutes}
For $k, \ell \ge 2$ arbitrary, the diagram
\[
\xymatrix{
K_{2k-1} \ar[r]^-{\phi_{2k-1}} \ar[d]	& G_{2k+1}/\pair{G_1,G_{2k}} \ar[d]\\
K_{2\ell-1} \ar[r]_-{\phi_{2\ell-1}} 	& G_{2\ell+1}/\pair{G_1,G_{2\ell}}
}
\]
commutes, where the vertical arrows are the evident isomorphisms induced from the isomorphisms $G_{2a + \epsilon} \cong G_{2b+\epsilon}$ for $a,b \ge 1$ and $\epsilon = 0,1$.
\end{lemma}

\begin{proof}
We must show that the formula \eqref{equation:2k+1} can be expressed entirely in terms of $\beta(A_{I_i})$ and $W_i \in G_{2k-1}$, at least modulo $\pair{G_1, G_{2k}}$. Of the terms in \eqref{equation:2k+1}, only the summands of the form $\pair{X_{ij},(\omega)_{(2k)}}$ do not visibly depend just on this data. We will use the Hermitian structure to see that they do.

Extending scalars from $\Z$ to $\frac{1}{2} \Z$, there is a decomposition
\[
M_n(\Z) = M_n(\Z)^+ \oplus M_n(\Z)^- \le M_n(\tfrac{1}{2} \Z)
\]
of matrices into their symmetric (resp. skew-symmetric) components; given $A \in M_n(\Z)$, we write
\[
A^\pm = \frac{A \pm A^T}{2}.
\]
As $(\alpha)_{(2k+1)} \in G_{2k+1}$ is symmetric, it suffices to see that $\pair{X_{ij},(\omega)_{(2k)}}^+$ can be expressed in terms of $\beta(A_{ij})$ and $W_i$ modulo $\pair{G_1,G_{2k}}$. 

$X_{ij}$ is itself symmetric, and from this one finds
\[
\pair{X_{ij},(\omega)_{(2k)}}^+ = \pair{X_{ij},(\omega)^-_{(2k)}}.
\]
Thus we are free to replace $(\omega)_{(2k)}$ with any skew-symmetric $\omega'$ satisfying
\[
\omega' \equiv (\omega)_{(2k)} \pmod{G_{2k}}.
\]
We will see that there is such an $\omega'$ whose entries depend only on $W_i$. 

Since $(\omega)_{(2k)}^-$ and $\omega'$ are skew-symmetric by assumption, the condition $\omega' \equiv (\omega)_{(2k)}^- \pmod{G_{2k}}$ is equivalent to the requirement that 
\[
((\omega)_{(2k)}^- - \omega') \vec 1 = 0,
\]
or by skew-symmetry,
\[
\vec 1 (\omega)_{(2k)} ^-= \vec 1 \omega'.
\]
From the equations $(\omega)_{(2k)} = (\omega)_{(2k)}^+ + (\omega)_{(2k)}^-$ and $\vec 1 (\omega)_{(2k)} = 0$, we see that
\[
\vec 1 \omega' = - \vec 1(\omega)_{(2k)}^+.
\]
To determine $\vec 1(\omega)_{(2k)}^+$, we look to the $s^{2k}$ term of $\omega^* J \omega - J$ which is zero by unitarity. From the equation $s^* = \frac{-s}{1+s}$ and the symmetry of $(\omega)_{(2k-1)}$, we find
\[
(\omega^*)_{(2k-1)} = -(\omega)_{(2k-1)}  \qquad (\omega^*)_{(2k)} = (\omega)_{(2k)}^T + (2k-1) (\omega)_{(2k-1)}.
\]
Extracting the $s^{2k}$ term of $(\omega^* J \omega - J) = 0$, 
\begin{equation}\label{equation:s2k}
\pair{(J)_{(1)},(\omega)_{(2k-1)}} + (\omega)_{(2k)}^T(J)_{(0)}  + (J)_{(0)} (\omega)_{(2k)} +(2k-1)(\omega)_{(2k-1)} (J)_{(0)} = 0.
\end{equation}
We recall from \eqref{equation:J0} that $(J)_{(0)} = 2 I - \bold{1}$ and that $\bold{1} (\omega)_{(2k-1)} = \bold{1} (\omega)_{(2k)} = 0$; inserting this into \eqref{equation:s2k} allows us to solve for $(\omega)_{(2k)}^+$:
\begin{equation}\label{equation:omegaplus}
(\omega)_{(2k)}^+ = -\tfrac{1}{4}(\pair{(J)_{(1)},(\omega)_{(2k-1)}}  + (4k-2)(\omega)_{(2k-1)}).
\end{equation}
The right-hand side of \eqref{equation:omegaplus} visibly depends only on $(\omega)_{(2k-1)} = W$ and allows one to express the row vector 
\[
u: = \vec 1 \omega'= - \vec1 (\omega)_{(2k)}^+
\]
in terms of $W$. 

From here it is a simple matter to construct a skew-symmetric $\omega'$ satisfying $\vec 1 \omega' = u$. Denote the entries of $u$ by $u_1, \dots, u_n$. Define
\[
(\omega')_{ij} = \begin{cases}
	u_1 + \dots + u_i	& i = j+1\\
	-(u_1 + \dots + u_i)	& j = i + 1\\
		0			& \mbox{otherwise}.
\end{cases}
\]
This is skew-symmetric by construction and visibly has the correct column sums for $1 \le i \le n-1$. In the final column, the sum is given by $-(u_1 + \dots + u_{n-1})$, and so it remains to show that $u_1 + \dots + u_n = 0$, or equivalently that $u \vec 1 = 0$. 

To see this, multiply the expression for $(\omega)_{(2k)}^+$ given in \eqref{equation:omegaplus} by the column vector $\vec 1$. As $(\omega)_{(2k-1)} \vec 1 = 0$, this gives
\[
(\omega)_{(2k)}^+ \vec 1= \tfrac{1}{4} (\omega)_{(2k-1)} (J)_{(1)}. 
\]
But multiplying this by the row vector $\vec 1$ now yields
\[
\vec 1 (\omega)_{(2k)}^+ \vec 1 = u \vec 1 = \vec 1\tfrac{1}{4} (\omega)_{(2k-1)} (J)_{(1)} = 0.
\]
\end{proof}

\section{Proof of Theorem \ref{theorem:image}}\label{section:proof}
In this final section we complete the proof of Theorem \ref{theorem:image} by establishing the equality $\mathfrak g(\Gamma) = \mathfrak g(\beta)$; indeed we show both of these are equal to $\mathfrak g$. We argue by induction on the degree of grading. The base cases are established in Lemma \ref{lemma:basecase}. To establish the inductive step, we perform the necessary calculation in Lemma \ref{lemma:calculation} and subsequently appeal to Lemma \ref{lemma:main}.

\begin{lemma}\label{lemma:basecase}
For $n \ge 4$ and $k \le 4$, there are surjections
\[
(\cdot)_{(k)}: B_n[s^k] \to G_{k}.
\]
\end{lemma}
\begin{proof}
The case $k = 0$ follows from the definition of $\Gamma$ given in Definition \ref{definition:gamma}: we restrict $\Gamma$ to consist only of those matrices that reduce mod $s$ to an element of $S_n$ precisely because the image of $\beta$ mod $s$ is $S_n$.

One computes directly from the definition of $\beta$ that $(\beta(A_{12}))_{(1)} = X_{12}$; it follows from this and the $S_n$-module structure that
\[
(\beta(A_{ij}))_{(1)} = X_{ij}
\]
for any $1 \le i < j \le n$. This establishes the claim in the case $k = 1$. The claim for $k = 2$ now follows by the case $k = 1$ of Lemma \ref{lemma:brackets}.

To establish the case $k = 3$, we use {\sc sage} 
to compute that
\begin{equation}\label{equation:specialalpha}
\alpha = [A_{13},A_{23}][A_{24},A_{14}][A_{14},A_{34}][A_{34},A_{24}]
\end{equation}
is contained in $\Gamma[s^3]$ and that 
\[
(\alpha)_{(3)} = X_{24}-X_{13}.
\]
The result now follows from Lemma \ref{lemma:gSn}. Finally $k = 4$ follows, like $k = 2$, from Lemma \ref{lemma:brackets}.
\end{proof}

\begin{lemma}\label{lemma:calculation}
For $n \ge 5$ and $k \ge 2$, the homomorphism $\phi_{2k-1}$ of Lemma \ref{lemma:main} is surjective.
\end{lemma}
\begin{proof}
Following Lemma \ref{lemma:commutes}, it suffices to consider the case $k=2$. This is established by a direct calculation. A computer search (again using {\sc sage}) shows that the element
\begin{equation}\label{equation:delta}
\delta = [\beta(A_{25}^2 A_{45}), [\alpha, \sigma_4]]
\end{equation}
with $\alpha$ as in \eqref{equation:specialalpha} is an element of $\Gamma[s^5]$ and satisfies
\[
(\delta)_{(5)} = \left( \begin{array}{ccccc} 0&2&0&2&-4\\ 2&-2&-2&1&1\\ 0&-2&0&-2&4\\ 2&1&-2&1&-2\\ -4&1&4&-2&1\end{array}\right)
\]
(this is illustrated here for $n = 5;$ for arbitrary $n$, $(\delta)_{(5)}$ is of course just extended by zeroes). Setting
\[
W = ([\alpha, \sigma_4])_{(3)} = X_{24}- X_{25}
\]
and
\[
d = (2 X_{25}+X_{45}) \otimes W \in K_3 \le G_1 \otimes G_3,
\]
the above calculation determines $\phi_3(d) := (\delta)_{(5)} \pmod{\pair{G_1,G_4}}$. Observe
\[
\left( \begin{array}{ccccc} 0&2&0&2&-4\\ 2&-2&-2&1&1\\ 0&-2&0&-2&4\\ 2&1&-2&1&-2\\ -4&1&4&-2&1\end{array}\right) = 
\left( \begin{array}{ccccc} 0&2&0&2&-4\\ 2&-2&-2&2&0\\ 0&-2&0&-2&4\\ 2&2&-2&0&-2\\ -4&0&4&-2&2\end{array}\right)+
\left( \begin{array}{ccccc} 0&0&0&0&0 \\ 0&0&0&-1&1\\ 0&0&0&0&0\\ 0&-1&0&1&0\\ 0&1&0&0&-1\end{array}\right).
\]
Lemma \ref{lemma:brackets} establishes the containment $2 G_5 \le \pair{G_1, G_4}$, so $\phi_3(d)$ is equal to the second summand $X_{24}-X_{25}$ modulo $\pair{G1,G4}$. By Lemma \ref{lemma:gSn},  this latter element generates $G_5$ as an $S_n$-module, and surjectivity of $\phi_3$ follows. 
\end{proof}

Finally we come to the induction step in the proof of Theorem \ref{theorem:image}.
\begin{lemma}\label{lemma:inductionstep}
For $n \ge 5$ and all $k \ge 0$, the homomorphism
\[
(\cdot)_{(k)}: B_n[s^k] \to G_{k}
\]
is surjective. Consequently there are equalities $\mathfrak g = \mathfrak g(\Gamma) = \mathfrak g(\beta)$, and by Proposition \ref{proposition:density}, Theorem \ref{theorem:image} holds.
\end{lemma}
\begin{proof}
Our inductive hypothesis is that the claim holds for all $j \le 2k$; Lemma \ref{lemma:basecase} establishes the base case $k =2$. Thus there is an element $\omega \in B_n[s^{2k-1}]$ with
\[
(\omega)_{2k-1} = X_{24}-X_{25}.
\]
Applying the calculation of Lemma \ref{lemma:calculation}, the element $[\beta(A_{25}^2A_{45}), \omega]$ is an element of $\beta[s^{2k+1}]$ and satisfies 
\begin{equation}\label{equation:isinimage}
([\beta(A_{25}^2A_{45}), \omega])_{(2k+1)} \equiv X_{24}-X_{25} \pmod{2G_{2k+1}}.
\end{equation}
By hypothesis, $(B_n[s^{2k-1}])_{(2k-1)} = G_{2k-1}$, and hence by the bracket formula of Lemma \ref{lemma:brackets}, 
\[
2G_{2k+1} \le (B_n[s^{2k+1}])_{(2k+1)}.
\]
Combining this with \eqref{equation:isinimage} shows that
\[
X_{24}-X_{25} \in (B_n[s^{2k+1}])_{(2k+1)}.
\]
Acting on $X_{24} - X_{25}$ by $S_n$, Lemma \ref{lemma:gSn}.\ref{item:odd} now shows that $(\beta[s^{2k+1}])_{(2k+1)} = G_{2k+1}$. The final assertion 
\[
(\beta[s^{2k+2}])_{(2k+2)} = G_{2k+2}
\]
follows from this and the bracket formula of Lemma \ref{lemma:brackets}. 
\end{proof}

\bibliography{burauimage}{}
\bibliographystyle{alpha}

\end{document}